\documentclass[a4paper,12pt]{article}

\usepackage[english]{babel}
\usepackage{hyperref}
\usepackage{amssymb,amsmath,amsthm,bm}
\usepackage{enumerate,enumitem}
\usepackage{graphicx,color}
\usepackage{epsfig}
\usepackage{tikz}
\usetikzlibrary{calc,decorations.pathmorphing,decorations.text,decorations.markings,matrix}
\usepackage{caption,subcaption}
\usepackage{refcount}
\usepackage[hmargin=2.5cm,vmargin=3cm]{geometry}
\usepackage{mathtools, braket}
\usepackage[normalem]{ulem}
\usepackage{authblk}
\usepackage{xcolor} 
\usepackage[square,sort,comma,numbers]{natbib}
\usepackage[thicklines]{cancel}
\usepackage{centernot}
\usepackage{indentfirst}
\usepackage[nameinlink]{cleveref}
\usepackage{booktabs}

\usepackage{array}
\allowdisplaybreaks[2]

\theoremstyle{plain}
\newtheorem{thm}{Thm}[section]

\newtheorem{claim}{Claim}[section]
\newtheorem{theorem}[thm]{Theorem}
\newtheorem{lemma}[thm]{Lemma}
\newtheorem{corollary}[thm]{Corollary}
\newtheorem{proposition}[thm]{Proposition}

\newtheorem{definition}[thm]{Definition}

\newenvironment{proof*}{\noindent \emph{Proof.}}{\hfill$\Diamond$}

\def\Z{\mathbb{Z}}

\def\SZ{S\Z}

\crefname{theorem}{Theorem}{Theorems}
\Crefname{theorem}{Theorem}{Theorems}

\title{Characterization of strongly $\mathbb{Z}_\ell$-connected graphs of small order}

\author{
{\normalsize 
			Jiaao Li\thanks{School of Mathematical Sciences and LPMC, Nankai University, Tianjin 300071, China; \texttt{lijiaao@nankai.edu.cn}}, 
			Bo Su\thanks{School of Mathematical Sciences and LPMC, Nankai University, Tianjin 300071, China; \texttt{suboll@mail.nankai.edu.cn}}, 
			Zhouningxin Wang\thanks{School of Mathematical Sciences and LPMC, Nankai University, Tianjin 300071, China; \texttt{wangzhou@nankai.edu.cn}},
			and Chunyan Wei\thanks{Department of Mathematical and Physical Sciences, Henan Institute of Engineering, Zhengzhou 451191, China;
            \texttt{yan1307015@163.com}}}}

\date{}

\begin{document}

\maketitle

\begin{abstract}
A graph is strongly $\Z_{\ell}$-connected if for each boundary function $\beta: V(G)\mapsto \Z_{\ell}$ with $\beta(v) \equiv d(v) \pmod{2}$ for every vertex $v$ and $\sum_{v \in V(G)} \beta(v) \equiv 0 \pmod{2\ell}$, there exists an orientation $D$ of $G$ such that $d_D^+(v) - d_D^-(v) \equiv \beta(v) \pmod{2\ell}$ for each $v \in V(G)$. This is a useful notion for studying circular flows of graphs. This note presents a fully self-contained, manual proof of a characterization of $4$-vertex strongly $\mathbb{Z}_\ell$-connected graphs for any integer $\ell\geq 2$, which will be used in our further study in this topic.
\end{abstract}

\section{Preliminaries}

The notion of strong $\mathbb{Z}_{\ell}$-connectivity for odd $\ell$ was studied in~\cite{LLLMMSZ14}, and was later extended to all integers $\ell$, including the even case; see~\cite{LLL17}.
When $\ell$ is odd, this general definition coincides with the one adopted in literature. This notion serves as a useful tool for studying circular flows of graphs.

This note provides a complete manual proof of the characterization of $4$-vertex strongly $\mathbb{Z}_{\ell}$-connected graphs for every integer $\ell$.  
In addition, we also give an alternative and simpler proof of the characterization of $3$-vertex strongly $\mathbb{Z}_{\ell}$-connected graphs in Section~\ref{sec:3-vertex-SZl}. 

We begin with some basic definitions and notation.
Let $G$ be a graph with vertex set $V(G)$. 
Denote by $v(G)$ and $e(G)$ the numbers of vertices and edges of $G$, respectively, and by $\delta(G)$ the minimum degree of $G$.  
For each $v \in V(G)$, let $d(v)$ denote its degree. 
A graph with exactly $k$ vertices is called a \emph{$k$-vertex graph}.

For disjoint sets $X,Y \subseteq V(G)$, let $[X,Y]_G$ denote the set of edges with one endpoint in $X$ and the other in $Y$. When $G$ is clear from context, we omit the subscript. 
If $X=\{x\}$, we write $[x,Y]$ instead of $[\{x\},Y]$; similarly, if $Y=\{y\}$, we write $[x,y]$. 
For each $X \subseteq V(G)$, let $d(X) := |[X,X^c]|$, which we call the \emph{degree of $X$}.

For each pair $x,y \in V(G)$, let $\mu_G(x,y) := |[x,y]_G|$, called the \emph{multiplicity} of $[x,y]$, that is, the number of parallel edges between $x$ and $y$. 
The \emph{multiplicity} of $G$ is defined as
\[
\mu(G) := \max\{\mu_G(x,y) : x,y \in V(G)\}.
\]

Given an orientation $D$ of $G$, let $E_D^+(v)$ (resp., $E_D^-(v)$) denote the set of arcs with tail (resp., head) at $v$, and let $d_D^+(v) := |E_D^+(v)|$ and $d_D^-(v) := |E_D^-(v)|$. 
When $D$ is clear from context, we omit the subscript.

Given an edge $e \in E(G)$, \emph{contracting} $e$ means identifying its endpoints and deleting the resulting loop; the resulting graph is denoted by $G/e$. 
More generally, if $H \subseteq G$, then contracting all edges of $H$ yields the graph $G/H$ (or $G/E(H)$). For a partition $\mathcal{P}$ of $V(G)$ (that is, a collection of pairwise disjoint nonempty subsets whose union is $V(G)$), let $|\mathcal{P}|$ denote the number of parts. 
The graph $G/\mathcal{P}$ is obtained by identifying all vertices in each part of $\mathcal{P}$ into a single vertex and deleting any resulting loops.

To \emph{lift} a path $P = v_0 v_1 \dots v_n$ in $G$, we delete all edges of $P$ and add a new edge $v_0 v_n$.

We also introduce two families of small graphs that will appear repeatedly. 
For an integer $a \ge 1$, let $aH$ denote the graph obtained from $H$ by replacing each edge with $a$ parallel edges. 
For integers $a,b,c \ge 0$, let $T_{a,b,c}$ denote the graph on three vertices in which each pair of vertices is joined by $a$, $b$, and $c$ parallel edges, respectively (see \Cref{fig:aK2,fig:Tabc}).

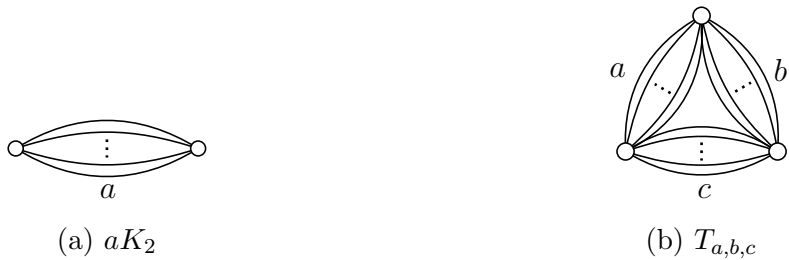
\begin{figure}[!htbp]
    \centering
	\begin{subfigure}[t]{.48\textwidth}
		\centering
		\begin{tikzpicture}[scale=0.4]		
			\draw [line width=1pt, dotted] (0, 0.3) to (0, -0.3);
			\draw(0.8, -1.4) node[left=0.5mm]  {$a$};
			\draw [bend left=18, line width=0.6pt, black] (-3,0) to (3,0);
			\draw [bend right=18, line width=0.6pt, black] (-3,0) to (3,0);
			\draw [bend left=32, line width=0.6pt, black] (-3,0) to (3,0);
			\draw [bend right=32, line width=0.6pt, black] (-3,0) to (3,0);
				
			\draw [fill=white,line width=0.6pt] (-3,0) node[left=0.5mm] { } circle (7pt);  
			\draw [fill=white,line width=0.6pt] (3,0) node[right=0.5mm] { } circle (7pt);  
			\end{tikzpicture}
			\caption{$aK_2$}
			\label{fig:aK2}     
		\end{subfigure}
		\begin{subfigure}[t]{.48\textwidth}
			\centering
			\begin{tikzpicture}[scale=0.4]			
				\draw [rotate=240] [line width=1pt, dotted] (1, -0.6) to (1, -1.2);
				\draw [bend left=20, line width=0.6pt, black] (0,2) to  (-2.5,-2.5) ;
				\draw [bend right=18, line width=0.6pt, black] (0,2) to  (-2.5,-2.5) ;
				\draw [bend left=34, line width=0.6pt, black] (0,2) to  (-2.5,-2.5) ;
				\draw [bend right=32, line width=0.6pt, black] (0,2) to  (-2.5,-2.5) ;

				\draw [line width=1pt, dotted] (0, -2.2) to (0, -2.8);
				\draw [bend left=20, line width=0.6pt, black]  (-2.5,-2.5)  to (2.5,-2.5);
				\draw [bend right=18, line width=0.6pt, black]  (-2.5,-2.5)  to (2.5,-2.5);
				\draw [bend left=34, line width=0.6pt, black]  (-2.5,-2.5)  to (2.5,-2.5);
				\draw [bend right=32, line width=0.6pt, black]  (-2.5,-2.5)  to (2.5,-2.5);

				\draw [rotate=120] [line width=1pt, dotted](-1, -0.7) to (-1, -1.3);
				\draw [bend left=20, line width=0.6pt, black] (0,2) to (2.5,-2.5); 
				\draw [bend right=18, line width=0.6pt, black] (0,2) to (2.5,-2.5);
				\draw [bend left=34, line width=0.6pt, black] (0,2) to (2.5,-2.5);
				\draw [bend right=32, line width=0.6pt, black] (0,2) to (2.5,-2.5);
				
				\draw [fill=white,line width=0.6pt] (0,2) node[above=0.5mm] { } circle (8pt); 
				\draw [fill=white,line width=0.6pt] (2.5,-2.5) node[right=0.5mm] { } circle (8pt);     
				\draw [fill=white,line width=0.6pt] (-2.5,-2.5) node[left=0.5mm] { } circle (8pt);   
				\draw  (-2,0.2) node[left=0.5mm] {$a$};   
				\draw(0.8,-3.8) node[left=0.5mm] {$c$};  
				\draw  (3.3,0.2) node[left=0.5mm] {$b$};  
			\end{tikzpicture}
			\caption{$T_{a,b,c}$}
			\label{fig:Tabc}
		\end{subfigure} 

    \caption{The graphs $aK_2$, $T_{a,b,c}$.}
    \label{fig:small graph with v(G)=4}
\end{figure}

\begin{definition}\label{def:(2ell,beta)-boundary_orientation}
Let $G$ be a graph.
\begin{enumerate}[label=(\arabic*), itemsep=0em]
	\item
	Let $\beta \colon V(G) \to \{0,1,\ldots,2\ell-1\}$ be a mapping.
	If $\beta(v) \equiv d(v) \pmod{2}$ for every vertex $v$ and
	\[
	\sum_{v \in V(G)} \beta(v) \equiv 0 \pmod{2\ell},
	\]
	then $\beta$ is called a \emph{parity-compliant $\mathbb{Z}_{2\ell}$-boundary} of $G$,
	or simply a \emph{$\mathbb{Z}_{2\ell}$-boundary}.

	\item
	Given a $\mathbb{Z}_{2\ell}$-boundary $\beta$ of $G$, an orientation $D$ of $G$
	is called a \emph{$\beta$-orientation} if
	\[
	d_D^+(v) - d_D^-(v) \equiv \beta(v) \pmod{2\ell}
	\]
	for every vertex $v$.

	\item
	The graph $G$ is \emph{strongly $\mathbb{Z}_{\ell}$-connected} if, for every
	$\mathbb{Z}_{2\ell}$-boundary $\beta$ of $G$, there exists a $\beta$-orientation
	of $G$.
\end{enumerate}
\end{definition}

For convenience, we write $G \in \mathrm{S}\mathbb{Z}_{\ell}$
to indicate that $G$ is strongly $\mathbb{Z}_{\ell}$-connected.
Such graphs satisfy several properties, some of which we list below. 

\begin{proposition}\label{prop:SZ5-property}
Fix a graph $G$ and a subgraph $H \subseteq G$.
Let $\ell$ be an integer with $\ell \ge 3$.
Then the following statements hold:
\begin{enumerate}[label=(\arabic*), itemsep=0em]
    \item
    $K_1 \in \mathrm{S}\mathbb{Z}_{\ell}$.

    \item\label{thm:SZl-contraction}\emph{\cite{L08,LLL17}}
    If $G \in \mathrm{S}\mathbb{Z}_{\ell}$, then both $G / e$ and $G + e'$
    belong to $\mathrm{S}\mathbb{Z}_{\ell}$ for each edge $e \in E(G)$
    and each new edge $e'$ with endpoints in $V(G)$ but $e' \notin E(G)$.

    \item\label{thm:SZl-contraction-reserve}\emph{\cite{L08,LLL17}}
    If $H \in \mathrm{S}\mathbb{Z}_{\ell}$ and $G / H \in \mathrm{S}\mathbb{Z}_{\ell}$,
    then $G \in \mathrm{S}\mathbb{Z}_{\ell}$.

    \item\label{thm:SZl-lifting}
    If $G'$ is obtained from $G$ by lifting certain paths and
    $G' \in \mathrm{S}\mathbb{Z}_{\ell}$, then $G \in \mathrm{S}\mathbb{Z}_{\ell}$.

    \item\label{thm:l-1spanning}\emph{\cite{LLL17}}
    If $G \in \mathrm{S}\mathbb{Z}_{\ell}$, then $G$ contains at least
    $\ell - 1$ edge-disjoint spanning trees.
    In particular, $\delta(G) \ge \ell - 1$ and
    $e(G) \ge (|V(G)| - 1)(\ell - 1)$.
\end{enumerate}
\end{proposition}

For graphs $G$ with $v(G) \le 4$, the following results were proved
in~\cite{LSWW24}.

\begin{theorem}[\cite{LSWW24}]\label{thm:small-SZl-other}
Let $\ell$ be an integer with $\ell \ge 3$.
\begin{enumerate}[label=(\arabic*), itemsep=0em]
\item\label{thm:SZl-small-2vertices}
The graph $aK_2$ belongs to $\SZ_\ell$ if and only if it contains
$\ell - 1$ edge-disjoint spanning trees, that is, if and only if
$a \ge \ell - 1$.

\item\label{thm:SZl-small-3vertices}
The graph $T_{a,b,c}$ belongs to $\SZ_\ell$ if and only if it contains
$\ell - 1$ edge-disjoint spanning trees, that is,
$a + b + c \ge 2\ell - 2$ and $\delta(G) \ge \ell - 1$.

\item\label{thm:SZl-small-4vertices}
Let $G$ be a graph with $v(G) = 4$.
If $e(G) \ge 3\ell - 2$, $\mu(G) \le \ell - 2$, and
$\delta(G) \ge \ell - 1$, then $G \in \SZ_\ell$.
\end{enumerate}
\end{theorem}

Note that if a graph $G$ with four vertices belongs to $\mathrm{S}\mathbb{Z}_{\ell}$,
then by Proposition~\ref{prop:SZ5-property}\ref{thm:l-1spanning},
it must satisfy $e(G) \ge 3\ell - 3$ and $\delta(G) \ge \ell - 1$.
Moreover, if $G$ contains a $\mu$-multiple edge $[x,y]$ with $\mu \ge \ell - 1$,
then by
Proposition~\ref{prop:SZ5-property}\ref{thm:SZl-contraction},
Proposition~\ref{prop:SZ5-property}\ref{thm:SZl-contraction-reserve},
and
Theorem~\ref{thm:small-SZl-other}\ref{thm:SZl-small-2vertices},
we have that
$G \in \mathrm{S}\mathbb{Z}_{\ell}$ if and only if
$G / [x,y] \in \mathrm{S}\mathbb{Z}_{\ell}$.
This reduces the problem to a smaller case that has already been settled by induction.
Therefore, it suffices to consider graphs $G$ with $\mu(G) \le \ell - 2$.

Consequently, Theorem~\ref{thm:small-SZl-other}\ref{thm:SZl-small-4vertices}
covers most cases, leaving only those with $e(G) = 3\ell - 3$ unresolved.
We now focus on this remaining case.

Before proceeding, we recall the following lemma from~\cite{LSWW24}.

\begin{lemma}[\cite{LSWW24}]
Let $\ell \ge 3$ be an integer and let $\beta$ be a $\mathbb{Z}_{2\ell}$-boundary
of a graph $G$.
Then there exists an integer-valued function
$\gamma \colon V(G) \to \{0, \pm1, \ldots, \pm(2\ell - 1)\}$
satisfying the following properties:
\begin{enumerate}[label=(\arabic*), itemsep=0em]
\item
For each vertex $v \in V(G)$,
\[
\gamma(v) \equiv \beta(v) \pmod{2\ell}
\quad \text{and} \quad
\gamma(v) \equiv d(v) \pmod{2}.
\]
\item
$
\sum_{v \in V(G)} \gamma(v) = 0.
$
\item
$
\max_{v \in V(G)} \gamma(v)
-
\min_{v \in V(G)} \gamma(v)
\le 2\ell.
$
\end{enumerate}
Such a function $\gamma$ is called the \emph{corresponding $\gamma$-function}
of $\beta$.
\end{lemma}

This lemma transforms the problem of finding an orientation with a
$\mathbb{Z}_{2\ell}$-boundary into the problem of finding an orientation
with an integer-valued boundary function that satisfies certain conditions.
By applying the following theorem of Hakimi~\cite{H65}, this task can be further reduced
to checking whether every vertex subset satisfies a degree condition.

\begin{theorem}[Hakimi~\cite{H65}]\label{thm:Hakimi}
Let $G$ be a graph and let $\gamma \colon V(G) \to \mathbb{Z}$ be a function
such that $\gamma(v) \equiv d(v) \pmod{2}$ for every vertex $v \in V(G)$
and $\sum_{v \in V(G)} \gamma(v) = 0$.
Then the following statements are equivalent:
\begin{enumerate}[label=(\arabic*), itemsep=0em]
\item There exists an orientation $D$ of $G$ such that
$d_D^+(v) - d_D^-(v) = \gamma(v)$ for all $v \in V(G)$.
\item For every subset $S \subseteq V(G)$, it holds that
$\left|\sum_{v \in S} \gamma(v)\right| \le d(S)$.
\end{enumerate}
\end{theorem}

As suggested in~\cite{LSWW24}, we refer to subsets violating
the second condition of Theorem~\ref{thm:Hakimi} as \emph{bad} subsets,
as formally defined below.
Note that under the following definition, if a subset $S \subseteq V(G)$
is bad, then so is its complement $S^c$.

\begin{definition}
Let $G$ be a graph, and let $\gamma \colon V(G) \to \{0, \pm1, \dots, \pm(2\ell - 1)\}$
be a function satisfying $\gamma(v) \equiv d(v) \pmod{2}$ for each $v \in V(G)$
and $\sum_{v \in V(G)} \gamma(v) = 0$.
A subset $S \subseteq V(G)$ is called \emph{$\gamma$-bad} (or simply \emph{bad}) if
\[
\left|\sum_{v \in S} \gamma(v)\right| > d(S).
\]
\end{definition}
		
Next, we require a result concerning graphs containing $m$ edge-disjoint spanning trees. 

\begin{theorem}[\cite{H24}]\label{thm:m-tree-connected}
Assume that $G$ contains $m$ edge-disjoint spanning trees and fix $z \in V(G)$.
If $d_G(z) \le 2m$, then there exist at most $d_G(z) - m$ disjoint pairs of non-parallel edges incident to $z$ such that, after lifting all the corresponding $2$-paths with $z$ as the internal vertex, the subgraph induced by $V(G) \setminus \{z\}$ in the resulting graph also contains $m$ edge-disjoint spanning trees.
\end{theorem}
	
We are now prepared to prove the remaining case where $v(G) = 4$
and $e(G) = 3\ell - 3$.

\section{Characterization of 4-vertex strongly $\mathbb{Z}_{\ell}$-connected graphs}

\begin{theorem}\label{thm:4-vertex-char}
Let $\ell \ge 3$ be an integer and let $G$ be a graph with
$v(G) = 4$, $e(G) = 3\ell - 3$, $\mu(G) \le \ell - 2$, and
$\delta(G) \ge \ell - 1$.
Then $G \notin \mathrm{S}\mathbb{Z}_{\ell}$ if and only if both of the following hold:
\begin{enumerate}[label=(\arabic*), itemsep=0em]
\item $d(S) = 2\ell - 2$ for every subset $S \subseteq V(G)$ with $|S| = 2$;
\item $d(v) \equiv \ell \pmod{2}$ for all $v \in V(G)$.
\end{enumerate}
Moreover, in this case, $G$ fails to admit a $\beta$-orientation
only for the unique $\mathbb{Z}_{2\ell}$-boundary
$\beta = ( \ell, \ell, \ell, \ell )^T$.
\end{theorem}
		
\begin{proof}
We first prove the necessary conditions.  
Suppose that $G$ is a graph as in the statement of the theorem, and that $G$ does not belong to $\SZ_\ell$.  
Let $\beta$ be a $\Z_{2\ell}$-boundary of $G$ for which $G$ has no $\beta$-orientation, and let $\gamma$ be any corresponding $\gamma$-function of $\beta$.  

Without loss of generality, label the vertices of $G$ as $v_1, v_2, v_3, v_4$ so that 
$$ \max_{v \in V(G)} \gamma(v)=\gamma(v_1) \ge \gamma(v_2) \ge \gamma(v_3) \ge \gamma(v_4) = \min_{v \in V(G)} \gamma(v).$$  
By the definition of a $\gamma$-function, we have $\gamma(v_1)-\gamma(v_4) \le 2\ell$.  
Moreover, by symmetry we may assume that $\gamma(v_1) \ge -\gamma(v_4)$, which immediately implies $-\ell \le \gamma(v_4) \le 0$.

In the following claims, we always assume that $\beta$ is this particular boundary and that the vertices are labeled as above.  
We now establish several claims regarding the structure of $G$ and the values of $\gamma$.

\begin{claim}\label{claim:edge-disjoint-trees}
$G$ contains $(\ell-1)$ edge-disjoint spanning trees. 
\end{claim}

\begin{proof}[Proof of Claim~\ref{claim:edge-disjoint-trees}]
Let $\mathcal{P}$ be an arbitrary partition of $V(G)$.

\begin{itemize}
\item If $|\mathcal{P}| = 4$, then $e(G/\mathcal{P}) = 3(\ell-1)$.
\item If $|\mathcal{P}| = 3$, then
\(
e(G/\mathcal{P}) \ge 3\ell - 3 - (\ell - 2) = 2\ell - 1 > 2(\ell-1).
\)
\item If $|\mathcal{P}| = 2$, then
\(
e(G/\mathcal{P}) \ge \min\{ 3\ell - 3 - 2(\ell - 2), \delta(G) \} \ge \ell - 1.
\)
\end{itemize}

By the Nash-Williams--Tutte Theorem, it follows that $G$ contains $(\ell-1)$ edge-disjoint spanning trees.
\renewcommand{\qedsymbol}{$\blacksquare$}
\end{proof}

\begin{claim}\label{claim:cut-size-2}
For any subset $S \subseteq V(G)$ with $|S| = 2$, we have $d(S) \ge \ell + 1$.
\end{claim}

\begin{proof}[Proof of Claim~\ref{claim:cut-size-2}]
By direct calculation, we have
\(
d(S) \ge 3\ell - 3 - 2(\ell - 2) = \ell + 1.
\)
\renewcommand{\qedsymbol}{$\blacksquare$}
\end{proof}
			
\begin{claim}\label{claim:no-single-vertex-bad}
No single vertex can form a $\gamma$-bad set. That is, for all $i \in [4]$, we have
\[
|\gamma(v_i)| \le d(v_i).
\]
\end{claim}

\begin{proof}[Proof of Claim~\ref{claim:no-single-vertex-bad}]
Suppose, for contradiction, that $|\gamma(v_j)| > d(v_j)$ for some $j \in [4]$.  
Since $d(v_j)$ and $\gamma(v_j)$ have the same parity, it follows that
\[
d(v_j) \le |\gamma(v_j)| - 2 \le 2\ell - 3 < 2(\ell - 1).
\]

Since $G$ contains $(\ell-1)$ edge-disjoint spanning trees and $d(v_j) < 2(\ell - 1)$, 
Theorem~\ref{thm:m-tree-connected} guarantees the existence of an integer $k \le d(v_j) - (\ell - 1)$
and $k$ disjoint pairs of non-parallel edges incident to $v_j$, such that after lifting the corresponding
$2$-paths with $v_j$ as the internal vertex, the subgraph induced by $V(G) \setminus \{v_j\}$ also contains $(\ell-1)$ edge-disjoint spanning trees.

Let $G_1$ denote the graph obtained from $G$ by performing these lifts, 
and let $G' \subseteq G_1$ be the subgraph induced by $V(G_1) \setminus \{v_j\}$. Thus, $G'$ contains $(\ell-1)$ edge-disjoint spanning trees. Hence, by Theorem~\ref{thm:small-SZl-other}\ref{thm:SZl-small-3vertices}, $G' \in S\mathbb{Z}_\ell$.

After lifts, the number of edges in $G'$ is
\[
e(G') = e(G) - d_G(v_j) + k = 3\ell - 3 - d_G(v_j) + k \ge 2\ell - 2,
\]
which forces $k = d_G(v_j) - (\ell - 1)$. Hence, the degree of $v_j$ in $G_1$ becomes
\[
d_{G_1}(v_j) = d_G(v_j) - 2k = 2(\ell - 1) - d_G(v_j).
\]

Observe that
\[
d_{G_1}(v_j) - (2\ell - |\gamma(v_j)|) = |\gamma(v_j)| - 2- d_G(v_j) \ge 0,
\]
so that
\[
d_{G_1}(v_j) \ge 2\ell - |\gamma(v_j)|, \quad \text{and} \quad
d_{G_1}(v_j) - (2\ell - |\gamma(v_j)|) \equiv 0 \pmod{2}.
\]

Therefore, in $G_1$, we can orient $2\ell - |\gamma(v_j)|$ edges into $v_j$ if $\gamma(v_j) \ge 0$, 
or away from $v_j$ if $\gamma(v_j) < 0$, and orient the remaining edges one-in-one-out.  
This ensures that the partial orientation at $v_j$ satisfies the boundary condition $\beta(v_j)$:
\[
d^+(v_j) - d^-(v_j) \equiv \gamma(v_j) \equiv \beta(v_j) \pmod{2\ell}.
\]

The orientation of the edges incident to $v_j$, together with the $\Z_{2\ell}$-boundary $\beta$ of $G$, naturally induces a $\Z_{2\ell}$-boundary $\beta'$ on $G'$, 
so that any $\beta'$-orientation of $G'$ extends to a $\beta$-orientation of $G_1$, and hence of $G$.  
Since $G' \in \mathrm{S}\mathbb{Z}_{\ell}$, it admits a $\beta'$-orientation, contradicting the assumption that $G$ has no $\beta$-orientation.

This completes the proof of the claim.
\renewcommand{\qedsymbol}{$\blacksquare$}
\end{proof}

\begin{claim}\label{claim:gamma-structure}
We have
\(
\gamma(v_3) < 0,
\)
and for any pair of vertices $v_i, v_j \in V(G)$,
\[
|\gamma(v_i) + \gamma(v_j)| \le 2\ell.
\]
\end{claim}

\begin{proof}[Proof of Claim~\ref{claim:gamma-structure}]
By \Cref{claim:no-single-vertex-bad}, any $\gamma$-bad set must consist of exactly two vertices.  

Suppose, for contradiction, that $\gamma(v_3) \ge 0$. Then we must have
\[
\gamma(v_1) \le |\gamma(v_4)| \le \ell.
\] 
By \Cref{claim:cut-size-2}, for any two-vertex subset $S = \{v_i, v_j\}$,
\[
|\gamma(v_i) + \gamma(v_j)| \le |\gamma(v_4)| \le \ell < \ell + 1 \le d(S),
\]
and so no two-vertex set can be $\gamma$-bad, a contradiction. Hence, 
\[
\gamma(v_3) < 0.
\]

Finally, since $- \ell \le \gamma(v_4) \le 0$ (as established at the start of the whole proof), it follows that for any pair of vertices $v_i, v_j \in V(G)$,
\[
|\gamma(v_i) + \gamma(v_j)| \le |\gamma(v_3) + \gamma(v_4)| \le 2\ell,
\]
as claimed.
\renewcommand{\qedsymbol}{$\blacksquare$}
\end{proof}

\begin{claim}\label{claim:gamma-pair-cut-bound}
For any two vertices $\{v_i, v_j\} \subseteq V(G)$, we have
\[
|\gamma(v_i) + \gamma(v_j)| \le d(\{v_i, v_j\})
\quad \text{or} \quad
2\ell - |\gamma(v_i) + \gamma(v_j)| \le d(\{v_i, v_j\}).
\]
Moreover, if $d(\{v_i, v_j\}) \ge 2\ell - 1$, then both inequalities hold.
\end{claim}

\begin{proof}[Proof of Claim~\ref{claim:gamma-pair-cut-bound}]
First, suppose $d(\{v_i, v_j\}) < 2\ell - 1$.  
If neither inequality holds, then
\[
|\gamma(v_i) + \gamma(v_j)| > d(\{v_i, v_j\}) \quad \text{and} \quad
2\ell - |\gamma(v_i) + \gamma(v_j)| > d(\{v_i, v_j\}).
\]  
Summing these two inequalities gives
\[
2\ell > 2d(\{v_i, v_j\}) \implies d(\{v_i, v_j\}) < \ell,
\]
which contradicts \Cref{claim:cut-size-2}.  
Hence at least one of the two inequalities holds.

Next, suppose $d(\{v_i, v_j\}) \ge 2\ell - 1$.  
Observe that $|\gamma(v_i) + \gamma(v_j)|$, $2\ell - |\gamma(v_i) + \gamma(v_j)|$, and $d(\{v_i, v_j\})$ all have the same parity.  
\begin{itemize}
\item If $d(\{v_i, v_j\})$ is even, then $d(\{v_i, v_j\}) \ge 2\ell$ and hence
\[
\max\bigl\{ |\gamma(v_i) + \gamma(v_j)|, \, 2\ell - |\gamma(v_i) + \gamma(v_j)|\bigr\} \le 2\ell \le d(\{v_i, v_j\}).
\]
\item If $d(\{v_i, v_j\})$ is odd, then $d(\{v_i, v_j\}) \ge 2\ell - 1$, and by parity, the maximum of the two expressions is at most $2\ell - 1 \le d(\{v_i, v_j\})$.
\end{itemize}
Thus, both inequalities hold in this case.
\renewcommand{\qedsymbol}{$\blacksquare$}
\end{proof}

\begin{claim}\label{claim:unique-gamma-bad-set}
The only $\gamma$-bad set is $\{v_1, v_2\}$ (equivalently, $\{v_3, v_4\}$).
\end{claim}

\begin{proof}[Proof of Claim~\ref{claim:unique-gamma-bad-set}]
We consider two cases based on the sign of $\gamma(v_2)$.

\noindent\textbf{Case 1:} $\gamma(v_2) \le 0$.  
Then we have 
\[
\gamma(v_1) > 0 \ge \gamma(v_2) \ge \gamma(v_3) \ge \gamma(v_4).
\] 
Since $\sum_{i=1}^4 \gamma(v_i) = 0$ and $\gamma(v_1) - \gamma(v_4) \le 2\ell$, it follows that
\[
0 \le \gamma(v_1) + \gamma(v_4) \le \gamma(v_1) + \gamma(v_3) \le \gamma(v_1) +\frac{1}{2}(\gamma(v_2) + \gamma(v_3)) = \frac{1}{2}(\gamma(v_1) - \gamma(v_4)) \le \ell.
\] 
Thus, the sets $\{v_1, v_3\}$ and $\{v_1, v_4\}$ cannot be $\gamma$-bad. Therefore, the only possible $\gamma$-bad set is $\{v_1, v_2\}$.

\noindent\textbf{Case 2:} $\gamma(v_2) > 0$.  
Then 
\[
\gamma(v_1) \ge \gamma(v_2) > 0 \ge \gamma(v_3) \ge \gamma(v_4).
\] 
Consider any set $S = \{v_i, v_j\}$ with $\gamma(v_i) > 0$ and $\gamma(v_j) < 0$, i.e., $i \in \{1,2\}$ and $j \in \{3,4\}$.  
By the bounds established at the start of the proof, we have
\[
-l \le \gamma(v_4) \le \min\Bigl\{ \sum_{v \in S} \gamma(v), \sum_{v \in S^c} \gamma(v) \Bigr\} \le 0.
\] 
This implies $|\gamma(v_i) + \gamma(v_j)| \le \ell$. By \Cref{claim:cut-size-2}, such a set $S$ cannot be $\gamma$-bad. Hence, the only possible $\gamma$-bad set is again $\{v_1, v_2\}$.
\renewcommand{\qedsymbol}{$\blacksquare$}
\end{proof}

\begin{claim}\label{claim:cut-size-2-exact}
For every 2-vertex subset $S \subseteq V(G)$, we have
\[
d(S) = 2\ell - 2.
\]
\end{claim}

\begin{proof}[Proof of Claim~\ref{claim:cut-size-2-exact}]
Suppose not. By \Cref{claim:unique-gamma-bad-set}, the only $\gamma$-bad set is $\{v_1, v_2\}$, so we must have
\[
\gamma(v_1) + \gamma(v_2) > d(\{v_1, v_2\}).
\]
Moreover, by \Cref{claim:gamma-structure}, $d(\{v_1, v_2\}) \le 2\ell - 2$.  
Since 
\[
d(\{v_1, v_2\}) + d(\{v_1, v_3\}) + d(\{v_1, v_4\}) = 2 e(G) = 6\ell - 6,
\] 
it follows that
\[
d(\{v_1, v_3\}) + d(\{v_1, v_4\}) = 4\ell - 4 + \bigl(2\ell - 2 - d(\{v_1, v_2\})\bigr) \ge 4\ell - 4.
\]

Hence, if some 2-vertex subset $S \subseteq V(G)$ satisfies $d(S) \neq 2\ell - 2$, then either
\[
d(\{v_1, v_3\}) \ge 2\ell - 1 \quad \text{or} \quad d(\{v_1, v_4\}) \ge 2\ell - 1.
\]

Without loss of generality, suppose $d(\{v_1, v_x\}) \ge 2\ell - 1$ for some $x \in \{3,4\}$, and let $\{y\} = \{3,4\} \setminus \{x\}$.  
Define a new function $\gamma'$ on $V(G)$ by
\[
\gamma'(v_1) = \gamma(v_1) - 2\ell, \quad
\gamma'(v_y) = \gamma(v_y) + 2\ell, \quad
\gamma'(v_i) = \gamma(v_i) \text{ for } i \in \{2,x\}.
\]
Clearly, $\gamma'(v) \equiv \gamma(v) \equiv \beta(v) \pmod{2\ell}$ for all $v \in V(G)$.  
We now verify that $G$ contains no $\gamma'$-bad set.

\begin{itemize}
\item \textbf{For the set $\bm{\{v_1, v_2\}}$:}  
Since $\gamma(v_1) + \gamma(v_2) > d(\{v_1, v_2\}) \ge 0$, \Cref{claim:gamma-pair-cut-bound} gives
\[
|\gamma'(v_1) + \gamma'(v_2)| = |\gamma(v_1) + \gamma(v_2) - 2\ell| \le d(\{v_1, v_2\}),
\] 
so $\{v_1, v_2\}$ is not $\gamma'$-bad.

\item \textbf{For the set $\bm{\{v_1, v_x\}}$:}  
If $x = 4$, then $\gamma(v_1) \ge -\gamma(v_4)$ by assumption.  
If $x = 3$, then $\gamma(v_1) + \gamma(v_3) < 0$ would imply $\gamma(v_2) + \gamma(v_4) \le \gamma(v_1) + \gamma(v_3) < 0$, contradicting $\sum_{i=1}^4 \gamma(v_i) = 0$.  
In either case, $\gamma(v_1) + \gamma(v_x) \ge 0$.  
Since $d(\{v_1, v_x\}) \ge 2\ell - 1$, \Cref{claim:gamma-pair-cut-bound} gives
\[
|\gamma'(v_1) + \gamma'(v_x)| = |\gamma(v_1) + \gamma(v_x) - 2\ell| \le d(\{v_1, v_x\}),
\] 
so $\{v_1, v_x\}$ is not $\gamma'$-bad.

\item \textbf{For the set $\bm{\{v_1, v_y\}}$:}  
We have
\[
|\gamma'(v_1) + \gamma'(v_y)| = |\gamma(v_1) + \gamma(v_y)| \le d(\{v_1, v_y\}),
\] 
so $\{v_1, v_y\}$ is also not $\gamma'$-bad.
\end{itemize}

Thus, $G$ contains no $\gamma'$-bad set. By Hakimi's theorem (Theorem~\ref{thm:Hakimi}), $G$ admits a $\gamma'$-orientation, and hence a $\beta$-orientation. This contradicts our assumption that $G$ admits no $\beta$-orientation.  

Therefore, every 2-vertex subset satisfies $d(S) = 2\ell - 2$, as claimed.
\renewcommand{\qedsymbol}{$\blacksquare$}
\end{proof}

\begin{claim}\label{claim:gamma-values-determined}
The $\gamma$-function satisfies 
\[
\gamma(v_1) = \gamma(v_2) = \ell \quad \text{and} \quad \gamma(v_3) = \gamma(v_4) = -\ell.
\]  
Consequently, for all $v \in V(G)$, we have
\[
\beta(v) \equiv \gamma(v) \equiv \ell \pmod{2\ell} \quad \text{and} \quad d(v) \equiv \ell \pmod{2}.
\]
\end{claim}

\begin{proof}[Proof of Claim~\ref{claim:gamma-values-determined}]
Since $\{v_1,v_2\}$ is a $\gamma$-bad set, we have
\[
\gamma(v_1) + \gamma(v_2) \ge d(\{v_1,v_2\}) + 2 = 2\ell.
\]  
On the other hand, by \Cref{claim:gamma-structure}, for any pair of vertices $v_i, v_j$ we have 
\(|\gamma(v_i) + \gamma(v_j)| \le 2\ell\).  
It follows that equality must hold, so
\[
\gamma(v_1) + \gamma(v_2) = 2\ell \quad \text{and} \quad \gamma(v_3) + \gamma(v_4) = -2\ell.
\]

Since $\gamma(v_3) \ge \gamma(v_4) \ge -\ell$, we deduce that $\gamma(v_3) = \gamma(v_4) = -\ell$.  
Moreover, the inequality $\gamma(v_1) - \gamma(v_4) \le 2\ell$ implies $\gamma(v_1) \le \ell$, and together with $\gamma(v_1) + \gamma(v_2) = 2\ell$ we conclude $\gamma(v_1) = \gamma(v_2) = \ell$.

Finally, since $\gamma(v) \equiv \beta(v) \equiv \ell \pmod{2\ell}$ and $d(v) \equiv \gamma(v) \pmod{2}$ for all $v \in V(G)$, we obtain the desired congruences $\beta(v) \equiv \ell \pmod{2\ell}$  and $ d(v) \equiv \ell \pmod{2}$.
\renewcommand{\qedsymbol}{$\blacksquare$}
\end{proof}

The content of Claims~\ref{claim:cut-size-2-exact} and~\ref{claim:gamma-values-determined} exactly matches the necessary conditions in the theorem, so the necessity is established.

We now prove the sufficiency. Suppose $G$ and a $\Z_{2\ell}$-boundary $\beta$ satisfy the conditions in the theorem. Assume, for contradiction, that $G$ admits a $\beta$-orientation $D$, and define $\gamma(v) := d_D^+(v) - d_D^-(v)$ for all $v \in V(G)$. By definition, $\gamma(v) \equiv \beta(v) \equiv \ell \pmod{2\ell}$, so $\gamma(v) \in \{\pm \ell, \pm 3\ell, \dots\}$. 

Hakimi's theorem (Theorem~\ref{thm:Hakimi}) requires $|\sum_{v \in S} \gamma(v)| \le d(S)$ for every subset $S \subseteq V(G)$. Applying this to singletons $S = \{v\}$, together with $d(v) \le e(G)= 3\ell - 3$, forces $\gamma(v) \in \{\pm \ell\}$ for all $v$. Ordering the vertices so that $\gamma(v_1) \ge \gamma(v_2) \ge \gamma(v_3) \ge \gamma(v_4)$, we recover $\gamma(v_1) = \gamma(v_2) = \ell$ and $\gamma(v_3) = \gamma(v_4) = -\ell$.

But then for $S = \{v_1, v_2\}$, we have 
$$
|\gamma(v_1) + \gamma(v_2)| = 2\ell > d(\{v_1, v_2\}) = 2\ell - 2,
$$
contradicting Hakimi's condition. So $G$ admits no $\beta$-orientation, finishing the proof.
\end{proof}

Now we present a complete characterization of graphs with four vertices that belong to $\SZ_\ell$.  
First, we introduce a simple reduction. Given a graph $G_0$, the \emph{$\SZ_\ell$-simplified graph} $G$ of $G_0$ is obtained by replacing any multiedge $[x,y]$ with multiplicity $\mu_0 \ge \ell-1$ by an $(\ell-1)$-multiedge.  
Observe that $G_0 \in S\mathbb{Z}_\ell$ if and only if $G_0/[x,y] \in S\mathbb{Z}_\ell$, which holds if and only if the $\SZ_\ell$-simplified graph $G$ belongs to $S\mathbb{Z}_\ell$.  

Based on this, we have the following characterization.

\begin{theorem}\label{thm:4-ver-character}
Let $G_0$ be a graph with $v(G_0) = 4$, and let $G$ be a $\SZ_\ell$-simplified graph of $G_0$.  
Then $G_0 \in S\mathbb{Z}_\ell$ if and only if $\delta(G) \ge \ell-1$, $e(G) \ge 3\ell-3$, and $G$ satisfies at least one of the following:
\begin{enumerate}[label=(\arabic*)]
    \item $\mu(G) = \ell-1$ or $e(G)\ge 3l-2$;
    \item there exists a 2-vertex set $S \subseteq V(G)$ with $d(S) \ne 2\ell-2$;
    \item there exists a vertex $v \in V(G)$ with $d(v) \not\equiv \ell \pmod{2}$.
\end{enumerate}
\end{theorem}

\begin{proof}
Since $G$ is $\SZ_\ell$-simplified, we have $\mu(G) \le \ell-1$. Clearly, the conditions $e(G) \ge 3\ell-3$ and $\delta(G) \ge \ell-1$ are necessary for $G \in S\mathbb{Z}_\ell$.  

We consider cases:

\begin{itemize}
    \item If $\mu(G) = \ell-1$, let $x,y \in V(G)$ be such that $\mu_G([x,y]) = \ell-1$. Then $G[x,y]$ is $\SZ_\ell$, and consider $G' = G/[x,y]$. We have
    \[
    e(G') = e(G) - (\ell-1) \ge 2\ell-2, \quad \delta(G') \ge \min\{\delta(G), 3\ell-3 - 2\mu(G)\} \ge \ell-1.
    \]
    By Theorem \ref{thm:small-SZl-other}\ref{thm:SZl-small-3vertices}, $G'$ is $\SZ_\ell$, and hence by Proposition \ref{prop:SZ5-property}\ref{thm:SZl-contraction-reserve}, $G$ is also $\SZ_\ell$.
    
    \item If $\mu(G) \le \ell-2$ and $e(G) \ge 3\ell-2$, then $G$ is $\SZ_\ell$ by Theorem \ref{thm:small-SZl-other}\ref{thm:SZl-small-4vertices}.
    
    \item If $\mu(G) \le \ell-2$ and $e(G) = 3\ell-3$, then the characterization is given precisely by Theorem \ref{thm:4-vertex-char}.
\end{itemize}

This covers all possibilities, completing the proof.
\end{proof}

When $\ell = 5$, the above theorem can be used to determine which $4$-vertex graphs belong to $S\Z_5$, which will be useful in proving that every planar graph with five edge-disjoint spanning trees is strongly $\mathbb{Z}_5$-connected. \Cref{thm:4-ver-character} leads directly to the following corollary.

\begin{corollary}\label{cor:4-vertex-SZ5}
Let $G$ be a graph with $v(G) = 4$, $e(G) = 12$, $\mu(G) \le 3$, and $\delta(G) \ge 4$.  
Then $G \notin S\Z_5$ if and only if $G$ is isomorphic to $W_1$ or $W_2$ (see Figure~\ref{fig:small graph with v(G)=4}).
\end{corollary}

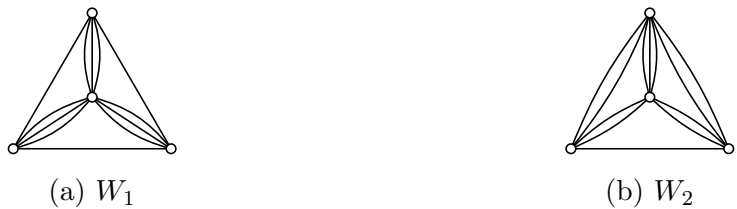
\begin{figure}[!htbp]
    \centering
	
    \begin{subfigure}[t]{0.45\textwidth}
        \centering
\begin{tikzpicture}[scale=0.6]			
    \draw [line width=0.6pt, black] (0,3) to (-1.73,0); 
    \draw [line width=0.6pt, black] (0,3) to (1.73,0); 
    \draw [line width=0.6pt, black] (-1.73,0) to (1.73,0); 

    \draw [line width=0.6pt, black] (0,1.14) to (1.73,0); 
    \draw [bend left=18,line width=0.6pt, black] (0,1.14) to (1.73,0); 
    \draw [bend right=18,line width=0.6pt, black] (0,1.14) to (1.73,0); 

    \draw [line width=0.6pt, black] (0,1.14) to (-1.73,0); 
    \draw [bend left=18,line width=0.6pt, black] (0,1.14) to (-1.73,0); 
    \draw [bend right=18,line width=0.6pt, black] (0,1.14) to (-1.73,0); 

    \draw [line width=0.6pt, black] (0,1.14) to (0,3); 
    \draw [bend left=18,line width=0.6pt, black] (0,1.14) to (0,3); 
    \draw [bend right=18,line width=0.6pt, black] (0,1.14) to (0,3); 

    \draw [fill=white,line width=0.6pt] (0,3) node[above] {} circle (3pt) ; 
    \draw [fill=white,line width=0.6pt] (-1.73,0) node[left] {} circle (3pt) ; 
    \draw [fill=white,line width=0.6pt] (1.73,0) node[right] {} circle (3pt) ; 
    \draw [fill=white,line width=0.6pt] (0,1.14) node[below=2mm] {} circle (3pt); 
\end{tikzpicture}
        \caption{$W_1$}
        \label{fig:W1}
    \end{subfigure}
    \begin{subfigure}[t]{0.45\textwidth}
        \centering
\begin{tikzpicture}[scale=0.6]			
    \draw [line width=0.6pt, black] (0,3) to (0,1.14); 
    \draw [bend left=16,line width=0.6pt, black] (0,3) to (0,1.14); 
    \draw [bend right=16, line width=0.6pt, black] (0,3) to (0,1.14); 
    
    \draw [bend right=8, line width=0.6pt, black] (0,3) to (-1.73,0); 
    \draw [bend left=8, line width=0.6pt, black] (0,3) to (-1.73,0); 
    
    \draw [bend right=8, line width=0.6pt, black] (0,3) to (1.73,0); 
    \draw [bend left=8, line width=0.6pt, black] (0,3) to (1.73,0); 

    \draw [bend right=10, line width=0.6pt, black] (0,1.14) to (1.73,0); 
    \draw [bend left=10, line width=0.6pt, black] (0,1.14) to (1.73,0); 

    \draw [bend right=10, line width=0.6pt, black] (0,1.14) to (-1.73,0); 
    \draw [bend left=10, line width=0.6pt, black] (0,1.14) to (-1.73,0); 

    \draw [line width=0.6pt, black] (-1.73,0) to (1.73,0); 

    \draw [fill=white,line width=0.6pt] (0,3) node[above] {} circle (3pt) ; 
    \draw [fill=white,line width=0.6pt] (-1.73,0) node[left] {} circle (3pt) ; 
    \draw [fill=white,line width=0.6pt] (1.73,0) node[right] {} circle (3pt) ; 
    \draw [fill=white,line width=0.6pt] (0,1.14) node[below=2mm] {} circle (3pt); 
\end{tikzpicture}
        \caption{$W_2$}
        \label{fig:W2}
    \end{subfigure}    
    \caption{Two graphs $W_1$, and $W_2$.}
    \label{fig:small graph with v(G)=4}
\end{figure}

\section{An alternative proof of the characterization of 3-vertex \texorpdfstring{$\SZ_\ell$}{SZ-l} graphs}
\label{sec:3-vertex-SZl}
In this section, we present a simpler proof of the characterization of $3$-vertex $\SZ_\ell$ graphs, which was previously given in \cite{LSWW24}.  
We begin with a combinatorial lemma.

\begin{lemma}\label{lem:intersection}
Let $\Omega$ be a finite set with $|\Omega| = k$, and let $I_1, \dots, I_m \subseteq \Omega$.  
Then
\[
|I_1 \cap \cdots \cap I_m|
\ge |I_1| + \cdots + |I_m| - (m-1)k.
\]
\end{lemma}

\begin{proof}
If $|I_1| + \cdots + |I_m| \le (m-1)k$, the inequality holds trivially.  
Thus, assume $|I_1| + \cdots + |I_m| > (m-1)k$.

We proceed by induction on $m$.  
For $m=2$, the statement follows directly from the inclusion--exclusion principle.

Assume the inequality holds for $m-1$ sets, and let $A = I_1 \cap \cdots \cap I_{m-1}$.  
Since $|I_m| \le|\Omega|\le k$, the assumption implies
$|I_1| + \cdots + |I_{m-1}| > (m-2)k$, and hence
$|A| \ge |I_1| + \cdots + |I_{m-1}| - (m-2)k$ by the induction hypothesis.
Consequently, $|A| + |I_m| > k$, which yields
$|A \cap I_m| \ge |A| + |I_m| - k$.
Substituting the lower bound for $|A|$ completes the proof.
\end{proof}

\begin{theorem}[\cite{LSWW24}]\label{thm:3-vertex-SZ-later}
A graph $G$ with $v(G)=3$ belongs to $S\mathbb{Z}_\ell$ if and only if it contains
$\ell-1$ edge-disjoint spanning trees. Equivalently,
$G \in S\mathbb{Z}_\ell$ if and only if $\delta(G)\ge \ell-1$ and
$e(G)\ge 2\ell-2$.
\end{theorem}

\begin{proof}
Let $V(G)=\{v_0,v_1,v_2\}$, and for each $i\in\{0,1,2\}$ (indices taken modulo $3$),
let $\mu_i$ denote the multiplicity of the edge between $v_i$ and $v_{i+1}$. For a subset $A\subseteq \mathbb{Z}_{2\ell}$ and $c\in \mathbb{Z}_{2\ell}$,
write $A+c:=\{a+c:\ a\in A\}$.

For each $i$, define
\[
J_i:=\{\,t\in\mathbb{Z}:\ -\mu_i\le t\le \mu_i,\ t\equiv \mu_i\pmod{2}\,\},
\]
and let
\[
I_i:=\{\,t\bmod 2\ell:\ t\in J_i\,\}\subseteq \mathbb{Z}_{2\ell}.
\]
Note that $|J_i|=\mu_i+1$, and hence
$|I_i|\ge \min\{\mu_i+1,\ell\}$.

Let $\beta$ be an arbitrary $\mathbb{Z}_{2\ell}$-boundary of $G$, and write
$\beta(v_i)=\beta_i$.
Then a $\beta$-orientation exists if and only if there exist elements
$x_i\in I_i$ satisfying
\[
x_i-x_{i-1}\equiv \beta_i \pmod{2\ell}
\qquad \text{for } i=0,1,2.
\]
These conditions are equivalent to the existence of an element
$x_0\in I_0\subseteq \Z_{2l}$ such that
\[
x_0\equiv x_1-\beta_1\equiv x_2+\beta_0 \pmod{2\ell}.
\]
Thus, it suffices to show that
\[
I_0\cap (I_1-\beta_1)\cap (I_2+\beta_0)\neq \varnothing .
\]

Since $\beta_i\equiv d(v_i)=\mu_i+\mu_{i-1}$, we have
\[
\mu_0\equiv \mu_1-\beta_1\equiv \mu_2+\beta_0 \pmod{2\ell}.
\]
Let
\[
\Omega:= I_0\cup (I_1-\beta_1)\cup (I_2+\beta_0)\subseteq \mathbb{Z}_{2\ell}.
\]
Then $|\Omega|\le \ell$.

Because $e(G)=\mu_0+\mu_1+\mu_2\ge 2\ell-2$, applying
Lemma~\ref{lem:intersection} yields
\begin{align*}
|I_0\cap (I_1-\beta_1)\cap (I_2+\beta_0)|
&\ge |I_0|+|I_1-\beta_1|+|I_2+\beta_0|-2\ell \\
&\ge \sum_{i=0}^2 \min\{\mu_i+1,\ell\}-2\ell \\
&\ge (2\ell-2)-(2\ell-3) \\
&\ge 1.
\end{align*}
Therefore, such an element $x_0$ exists, and hence $G$ admits a
$\beta$-orientation.
\end{proof}

\bibliographystyle{abbrv}
\bibliography{reference}

\end{document}